\documentclass{amsart}
\usepackage{verbatim,amsfonts,color,mathrsfs,stmaryrd, bbold}
\usepackage[all]{xy}
\usepackage{amssymb,amsmath,amscd,epsf,xypic,amsxtra,graphicx, url, caption, subcaption} 
\usepackage{tikz-cd}
\usepackage{tikz}
\usetikzlibrary{arrows}

\theoremstyle{plain}
\newtheorem{Thm}{Theorem}
\newtheorem{Cor}{Corollary}

\newtheorem{Lem}{Lemma}
\newtheorem{Prop}{Proposition}

\theoremstyle{definition}

\newtheorem{Rmk}{Remark}

\theoremstyle{remark}

\begin{document}

 
\title[SAF-zero pseudo-Anosov maps]{New infinite families of  pseudo-Anosov maps with vanishing Sah-Arnoux-Fathi invariant}

\author{Hieu Trung Do} 
\address{Oregon State University\\Corvallis, OR 97331}
\email{doh@math.oregonstate.edu}

\author{Thomas A. Schmidt}
\address{Oregon State University\\Corvallis, OR 97331}
\email{toms@math.orst.edu}
\date{4 March 2016}


\begin{abstract}   We show that an orientable pseudo-Anosov homeomorphism has vanishing Sah-Arnoux-Fathi invariant if and only if the minimal  polynomial of its  dilatation is  not  reciprocal.   We relate this to works of Margalit-Spallone and Birman, Brinkmann and Kawamuro.   Mainly, we use Veech's construction of pseudo-Anosov maps to give explicit pseudo-Anosov maps of vanishing Sah-Arnoux-Fathi invariant.  In particular, we give new infinite families of such maps in genus 3.  
\end{abstract}

\maketitle

\section{Introduction}  

In 1981,  Arnoux-Yoccoz \cite{AY} gave the first example of a pseudo-Anosov homeomorphism whose dilatation was of degree less than twice the genus of the surface on which it is defined.  In fact, they gave an infinite family of these, one in each genus $g \ge 3$.   In his Ph.D. dissertation of the same year,  Arnoux \cite{A2}, see also \cite{A}, showed that each of these maps has vanishing Sah-Arnoux-Fathi (SAF) invariant.    Pseudo-Anosov maps with vanishing SAF-invariant are especially interesting for their dynamical properties,  see   \cite{ABB, LPV, LPV2, McCasc}.  However,  there are few examples known, see below for a list of these.    We find a new infinite family, and to aid in the search for these interesting maps, we also clarify criteria in the literature derived from work of Calta-Smillie \cite{CS}.

We characterize pseudo-Anosov maps with vanishing SAF-invariant. 
\begin{Thm}\label{t:characterization}   Suppose that $\phi$ is an orientable pseudo-Anosov map of a closed compact surface, with  dilatation $\lambda$.   Then $\phi$ has vanishing Sah-Arnoux-Fathi invariant if and only if the minimal polynomial of $\lambda$ is not reciprocal. 
\end{Thm}

  We  give explicit constructions of new infinite families of pseudo-Anosov maps with vanishing SAF-invariant.  

\begin{Thm}\label{t:newFamily2}   For each $k \in \mathbb N$ with $k\ge 2$, there exists at least four orientable pseudo-Anosov maps in the hyperelliptic component of the stratum $\mathscr{H}(2,2)$ having dilatation of minimal polynomial $x^3-(2k+4)x^2+(k+4)x-1$. In particular, each of these pseudo-Anosov maps has vanishing \textnormal{SAF}-invariant.
\end{Thm}

In Subsection \S~\ref{ss:biPerronHomol} we apply a  construction of pseudo-Anosov homeomorphisms given by Margalit-Spallone \cite{MS}   to lend support to a conjecture about the set of all dilatations of pseudo-Anosov homeomorphisms.   Recall that a real  algebraic number  $\alpha$ greater than one is called {\em bi-Perron} if all of its conjugates (other than itself) lie in the annulus $\{ ||\alpha||^{-1} \le ||z|| < ||\alpha||\}$, where   $||z||$ denotes the norm of a complex number.   An algebraic integer, thus having minimal polynomial with integer coefficients,  is a {\em unit} if its inverse is also an algebraic integer.  Fried \cite{F} showed that the dilatation of any pseudo-Anosov map is a bi-Perron unit.  A conjecture that Farb-Margalit \cite{FM}  attribute to C.~McMullen (and is a question in \cite{F}), states that every bi-Perron unit is the dilatation of some pseudo-Anosov homeomorphism.     Recall that exactly when a pseudo-Anosov homeomorphism is orientable,  its dilatation is an eigenvalue of the homeomorphism's induced action on first integral homology.    The construction of Margalit-Spallone \cite{MS} shows that any polynomial that passes a certain homological criterion,  see below,  is the characteristic polynomial of the homology action induced by some pseudo-Anosov map.     Using this, we find a partial confirmation of the conjecture.

\bigskip
In Subsection~\S \ref{ss:BirmanQ},  we answer an implicit question of Birman, Brinkmann and Kawamuro   \cite{BBK}.   Namely,  if $\phi$ is an orientable pseudo-Anosov map on a genus $g$ compact surface without punctures, then their symplectic polynomial $s(x)$ associated to $\phi$ is  reducible if and only if either $\phi$ has vanishing SAF-invariant, or $\phi$ has trace field of degree less than $g$. 

\section{Background}

\subsection{Pseudo-Anosov map,  translation surface} Suppose that $X$ is an orientable closed real surface   of genus $g\ge 2$.
The Teichm\"uller modular group $\text{Mod}(X)$ is the quotient of the group of orientation preserving homeomorphisms by the subgroup of those homeomorphisms isotopic to the identity.   A mapping class $[\phi] \in \text{Mod}(X)$ is called {\em pseudo-Anosov},
if there exists a representative $\phi : X \to X$, a pair of invariant transverse measured (singular) foliations
$(\mathcal F^u, \mu^u), (\mathcal F^s, \mu^s)$, and a real number $\lambda$, the dilatation of $[\phi]$, such that $\phi$ multiplies
the transverse measure $\mu^u$ (resp. $ \mu^s$) by $\lambda$ (resp. $\lambda^{-1}$). The real number $\lambda = \lambda(\phi)$ is called the {\em dilatation} of the {\em pseudo-Anosov homeomorphism} $\phi$.  Some prefer to call $\lambda$ the {\em stretch factor} of $\phi$.   

A pseudo-Anosov homeomorphism  $\phi$ is called {\em orientable} if either of (and hence both)  $\mathcal F^u$ or $\mathcal F^s$ is orientable (that is,  leaves can be consistently oriented).    As \cite{LT} recall (see their Theorem 2.4),  a pseudo-Anosov homeomorphism $\phi$ is orientable if and only if its dilatation is an eigenvalue of the standard induced action on first homology $\phi_*: H_1(X, \mathbb Z) \to H_1(X, \mathbb Z)$. 

Orientability of either foliation is equivalent to  every singularity having an even number of prongs.   
Indeed, by Hubbard-Masur \cite{HM} the pair of measured foliations defines  a quadratic differential and a complex structure on $X$ so that this quadratic differential is holomorphic, orientability corresponds to the quadratic differential being the square of a holomorphic 1-form (thus, an {\em abelian differential}), say $\omega$.  Fixing base points and integrating $\omega$ along paths  defines local coordinates on $X$ (in $\mathbb C$ or $\mathbb R^2$, depending on our need),  transition functions are by translations, and the result is a {\em translation surface}, $(X, \omega)$.  (The aforementioned singularities occur at the zeros of $\omega$.)  The pseudo-Anosov $\phi$ acts affinely with respect to the local Euclidean structure of $(X, \omega)$.    Furthermore, taking the view of real local coordinates,   $\text{SL}_2(\mathbb R)$ acts on the collection of all translation surfaces  by post-composition with the local coordinate maps.   

We often use the words {\em pseudo-Anosov map} to mean an orientable pseudo-Anosov homeomorphism (usually with an emphasis on its translation surface).

\subsection{SAF-zero defined}   The Sah-Arnoux-Fathi (SAF) invariant was first defined for any interval exchange transformation (for more on these interval maps, see Subsection~\ref{ss:veechCon}). Given $f$ defined  on an interval $I = \cup_{j=1}^{d}\l,I_j$  and given piecewise by $f(x) = x + \tau_j$ on $I_j$,  the SAF-invariant of $f$ is the element of $\mathbb R\wedge_{\mathbb Q}\,\mathbb R$ given by $\sum_{j=1}^d\, \lambda_j \wedge \tau_j$, where $\lambda_j$ is the length of $I_j$.   In \cite{A2}, Arnoux showed that any linear flow on a translation surface defines a family of interval exchange maps, by taking 
any appropriately chosen full transversal of the flow, all having the same SAF-invariant.     One says that a pseudo-Anosov map has vanishing SAF-invariant if the flow in its stable direction has its first return interval exchange transformations with this property.  (Below we will show that this is then also true of the flow in the unstable direction.)
   
\subsection{Known examples}  Besides the Arnoux-Yoccoz family of SAF-zero pseudo-Anosov maps (one per genus at least three),  the other known infinite families are the Arnoux-Rauzy family in genus three discussed in \cite{LPV2},  and the examples of Calta and Schmidt \cite{CS} found by Fuchsian group techniques.    Sporadic examples were given by Arnoux-Schmidt \cite{AS} and in \cite{CS};   McMullen \cite{McCasc} presents an example in genus three found by Lanneau.

\subsection{Trace field, periodic direction field, Veech group} \label{ss:especiallyCaltaSmillie}  The {\em trace field}  of the translation surface $(X, \omega)$ of a pseudo-Anosov map of dilatation $\lambda$ coincides with $k = \mathbb Q(\lambda + \lambda^{-1})$, see the appendix of \cite{KS}.     If a translation surface has at least three directions of vanishing SAF-invariant, then Calta and Smillie  \cite{CSm} show that the surface can be normalized by way of the $\text{SL}_2(\mathbb R)$-action so that the directions with slope $0$, $1$ and $\infty$ have vanishing SAF-invariant.  They further  prove that on the normalized surface the set of slopes of directions  with vanishing SAF-invariant forms a field (union with infinity, thus more precisely the projective line over the field).  A translation surface so normalized is said to be in {\it standard form}, and the field so described is called the {\it periodic direction field}.  Calta-Smillie also show that when $(X, \omega)$ arises from a pseudo-Anosov map, then it can be placed in standard form,  and more importantly its trace field and periodic direction field coincide.     

The {\em Veech group} $\text{SL}(X, \omega) \subset \text{SL}_2(\mathbb R)$ is the group of matrix parts of (orientation-preserving) affine diffeomorphisms of $(X, \omega)$.  An affine diffeomorphism  of $(X, \omega)$ is  pseudo-Anosov if and only if its 
matrix part  is a hyperbolic element of $\text{SL}_2(\mathbb R)$, see \cite{T, V}.   Furthermore, if there is any such pseudo-Anosov map, then $\text{SL}(X, \omega) \subset \text{SL}_2(k)$, where $k$ is the trace field (this follows from the appendix of \cite{KS}:   the trace field is also the holonomy field and elements of the Veech group preserve the two dimensional $k$-vector space spanned by the holonomy vectors;  the statement also follows from Theorem~1.5 of \cite{CSm}).

\subsection{Homological criterion, Margalit-Spallone construction}  Margalit and Spallone \cite{MS} give a construction of  pseudo-Anosov classes in the Teichm\"uller modular group.  
Recall that a polynomial $p(x)=\sum_{i=0}^{n} c_i x^i$ is called \textit{reciprocal} when $c_i=c_{n+1-i}$ for all $i=1,...,n$.    (The characteristic polynomial of any symplectic matrix is monic reciprocal.) A monic reciprocal polynomial with integral coefficients is called {\em symplectically irreducible} if it is not the product of reciprocal polynomials of strictly lesser degree. 

  The {\em homological criterion} for a monic reciprocal polynomial $q(x)$ of even degree is that all of the following hold.
\begin{itemize}
\item $q(x)$ is symplectically  irreducible,
\item $q(x)$ is not cyclotomic, and 
\item $q(x)$ is not a polynomial in $x^k$ for any integral $k>1$.
\end{itemize}

Margalit-Spallone extend a result of  Casson-Bleiler:   
For any $f$ representing a class of the modular group of a closed surface $X$ of genus at least two,  let $q_f(x)$ be the characteristic polynomial for the action on first integral homology induced by $f$.   If $q_f(x)$ passes the homological criterion, then the class of $f$ is pseudo-Anosov.   Furthermore, by considering words in explicit elements of the modular group,  for any $q(x)$ passing the homological criterion Margalit-Spallone  build a homeomorphism $f$ whose homological action has characteristic polynomial $q(x)$.  Hence the class of $f$ (and indeed all of its Torelli group coset) is pseudo-Anosov.

\subsection{Veech construction}\label{ss:veechCon}   In this subsection we mainly reproduce Lanneau's \cite{L} overview (following \cite{MMY}) of Veech's construction of pseudo-Anosov homeomorphisms using the Rauzy-Veech induction,  \cite{V}.   In this subsection we follow standard convention and let $\lambda$ denote  the length vector for an interval exchange transformation. 
	
\subsubsection{Interval Exchange Transformation}
	An {\em interval exchange transformation} (IET) is a one-to-one map $T$ from an open interval $I$ to itself that permutes, by translation, a finite partition ${I_j , j = 1, . . . , d}$ of $I$ into $d \geq 2$ open subintervals. It is easy to see that $T$ is precisely determined by the following data: a permutation $\pi$ that encodes how the intervals are exchanged, and a vector $\lambda$ with positive entries that encodes the lengths of the intervals.
	
	It is useful to employ a redundant notation for IETs.     A permutation is a pair of one-to-one maps $(\pi_0, \pi_1)$ from a finite alphabet $\mathscr{A}$ to ${1, . . . , d}$ in the following way. In the partition of $I$ into intervals, we denote the interval labeled $k$, when counted from the left to the right, by $I_{\pi_0^{-1}(k)}$. Once the intervals are exchanged, the interval labeled $k$ is $I_{\pi_1^{-1}(k)}$. The permutation $\pi$ corresponds to the map $\pi = \pi_1 \circ \pi_{0}^{-1}$. The lengths of the intervals form a vector $\lambda = (\lambda_\alpha)$, $\alpha \in \mathscr{A}$. We will usually represent the combinatorial datum $\pi = (\pi_0, \pi_1)$ by a table:
\[ 
	\pi = 	
	\begin{pmatrix}
	\pi_0^{-1}(1) & \pi_0^{-1}(2) & ... & \pi_0^{-1}(d)\\ 
	\pi_1^{-1}(1) & \pi_1^{-1}(2) & ... & \pi_1^{-1}(d)
	\end{pmatrix}.
\]

It is reasonable to focus on  those IET that cannot trivially be decomposed into two distinct IETs.  For this,  a permutation $\pi$ is called reducible if for any $1 \le  k<d$,  $\pi_0^{-1}(\{1, \dots, k\}) = \pi_1^{-1}(\{1, \dots, k\})$.   Otherwise,  $\pi$ is called {\em irreducible}.  

\subsubsection{Suspension data}	
	A suspension datum for $T$ is a complex vector $\zeta$ of length $d$ such that
\begin{enumerate}	
	\item $\forall\alpha \in \mathscr{A}$, $\Re(\zeta_\alpha) = \lambda_\alpha$.\\
	\item $\forall 1 \leq k \leq d-1$, $\Im(\Sigma_{\pi_0(\alpha)\leq k} \zeta_\alpha)>0$. \\
	\item$\forall 1 \leq k \leq d-1$, $\Im(\Sigma_{\pi_1(\alpha)\leq k} \zeta_\alpha)<0$.
\end{enumerate}	
	
	To each suspension datum $\zeta$, we can associate a translation surface $(X, \omega) = X(\pi, \zeta)$ in the following way. Consider
the broken line $L_0$ on $\mathbb{C} = \mathbb{R}^2$ defined by
concatenation of the vectors $\zeta_{\pi_0^{-1}(j)}$ (in
this order) for $j = 1, . . . , d$ with starting point at the origin. Similarly, we consider the broken line $L_1$ defined by concatenation of the vectors $\zeta_{\pi_1^{-1}(j)}$ (in this order) for $j = 1, . . . , d$ with starting point at the origin. If the lines $L_0$ and $L_1$ have no intersections other than the endpoints, we can construct a translation surface $X$ by identifying each side $\zeta_j$ on $L_0$ with the side $\zeta_j$ on $L_1$ by a translation. The resulting surface is a translation surface endowed with the form $\omega = dz$. 

Let $I \subset X$ be the horizontal interval defined by $I = (0, \Sigma_\alpha \lambda_\alpha) \cup \{0\}$. Then the interval exchange map $T$ is precisely the one defined by the first return map to
$I$ of the vertical flow on $X$.

\subsubsection{Rauzy-Veech induction}
	The Rauzy-Veech induction $\mathscr{R}(T )$ of $T$ is defined
as the first return map of $T$ to a certain subinterval $J$ of $I$. 	
	We recall very briefly the construction. The type of $T$ is 0 if $\lambda_{\pi_0^{-1}(d)} > \lambda_{\pi_1^{-1}(d)}$ and 1 if $\lambda_{\pi_1^{-1}(d)} > \lambda_{\pi_0^{-1}(d)}$.  We say that the letter $\pi_0^{-1}(d)$ is the {\em winner} of this induction step, or that $\pi_0^{-1}(d)$ is, respectively.  We define a subinterval $J$ of $I$ by
	\[J=   \left\{
\begin{array}{ll}
      I \backslash T(\lambda_{\pi_1^{-1}(d)}) & \text{if $T$ is of type 0}; \\
      
      I \backslash \lambda_{\pi_0^{-1}(d)} & \text{if $T$ is of type 1}. \\
\end{array} 
\right. \]
The image of $T$ by the Rauzy-Veech induction $\mathscr{R}$ is defined as the first return map of $T$ to the subinterval $J$. This is again an interval exchange transformation, defined on $d$ letters. Thus this defines two maps $\mathscr{R}_0$ and $\mathscr{R}_1$ by $\mathscr{R}(T ) = (\mathscr{R}_\epsilon(\pi), \lambda')$, where $\epsilon$ is the type of $T$. The new
data and transition matrix are found as follows. 

\begin{enumerate}
\item If  $T$ has type 0, let $k$ be $\pi_1^{-1}(k)=\pi_0^{-1}(d)$ with $k\leq d-1$. Then $\mathscr{R_0}(\pi_0, \pi_1) = (\pi_0',\pi_1')$ where $\pi_0=\pi_0'$ and
	\[\pi_1'^{-1}(j)=   \left\{
\begin{array}{ll}
      \pi_1^{-1}(j) & \text{if $j \leq k$}; \\
      \pi_1^{-1}(d) & \text{if $j=k+1$};\\
      \pi_1^{-1}(j-1) & \text{otherwise}. \\
\end{array} 
\right. \]

\item If  $T$ has type 1, let $k$ be $\pi_0^{-1}(k)=\pi_1^{-1}(d)$ with $k\leq d-1$. Then $\mathscr{R_1}(\pi_0, \pi_1) = (\pi_0',\pi_1')$ where $\pi_1=\pi_1'$ and
	\[\pi_0'^{-1}(j)=   \left\{
\begin{array}{ll}
      \pi_0^{-1}(j) & \text{if $j \leq k$}; \\
      \pi_0^{-1}(d) & \text{if $j=k+1$};\\
      \pi_0^{-1}(j-1) & \text{other wise}. \\
\end{array} 
\right. \]

\item  The new lengths $\lambda'$ and $\lambda$ are related by a positive transition matrix $V_{\alpha\beta}$ with $V_{\alpha\beta}\lambda'=\lambda$.

\medskip

\item  If $T$ is of type 0 then let $(\alpha,\beta)=(\pi_0^{-1}(d),\pi_1^{-1}(d))$ otherwise let $(\alpha,\beta)=(\pi_1^{-1}(d),\pi_0^{-1}(d))$. With this notation,  $V_{\alpha\beta}$ is the  matrix $I+E_{\alpha\beta}$ where $E_{\alpha\beta}$  is the matrix whose only nonzero entry is at $(\alpha,\beta)$ where the value is 1.
\end{enumerate}

\medskip 
Iterating the Rauzy-Veech induction $n$ times, we obtain a sequence of transition matrices $\{V_k\}$. We can write $\mathscr{R}^{(n)}(\pi,\lambda)=(\pi^{(n)},\lambda^{(n)})$ with $(\Pi^n_{k=1}V_k)\lambda^{(n)}=\lambda$.

We can also define the Rauzy-Veech induction on the space of suspensions by 
\[	\mathscr{R}(\pi,\zeta) = (\mathscr{R}_\epsilon\pi, V^{-1}\zeta), \, \text{where}\; V =  \Pi^n_{k=1}V_k\,.
\]

If $(\pi',\zeta')=\mathscr{R}(\pi,\zeta)$ then the two translation surfaces $X(\pi,\zeta)$ and $X(\pi',\zeta')$ are isometric,  i.e. they define the same surface in the moduli space.

For a combinatorial datum $\pi$,  we call the {\em  Rauzy class} of $\pi$ the set of all combinatorial data that can be obtained from $\pi$ by the combinatorial Rauzy moves.  The   {\em labeled Rauzy diagram} of $\pi$ is the directed 
graph whose vertices are all combinatorial data that can be obtained from $\pi$ by the combinatorial Rauzy moves. From each vertices, there are two directed outgoing edges labeled 0 and 1 (the type) corresponding to the two combinatorial Rauzy moves.

\subsubsection{Closed loops and pseudo-Anosov homeomorphisms}  We are now ready to describe Veech's construction of pseudo-Anosov homeomorphisms. Let $\pi$ be an irreducible permutation and let $\gamma$ be a closed loop in the Rauzy diagram associated to $\pi$.   We obtain the matrix $V$ as above; let us assume that $V$ is primitive (i.e. there exists $k$ such that for all $i$, $j$, the $(i, j)$ entry of $V^k$ is positive) and let $\theta > 1$ be its Perron-Frobenius eigenvalue. We choose a positive eigenvector $\lambda$ for $\theta$. Now,  $V$ is appropriately symplectic (see \cite{Vi} for an explanation of this result of \cite{V}), allowing one to choose  $\tau$ an eigenvector for the eigenvalue $\theta^{-1}$ with $\tau_{\pi_0^{-1}(1)}>0$.  We form the
vector $\zeta =  \lambda + i \tau$. We can show that $\zeta$ is a suspension data for $\pi$. Thus, with a minor abuse of notation, 
\[
	\mathscr{R}(\pi,\zeta)=(\pi,V^{-1}\zeta)= (\pi, V^{-1}\lambda,V^{-1}\tau)=(\pi,\theta^{-1}\lambda,\theta\tau)=g_t(\pi,\lambda,\tau)
\]
where $t=\log(\theta)>0$.

	The two surfaces $X(\pi, \zeta)$ and $g_tX(\pi, \zeta)$ differ by some element of the mapping class group. In other words there exists a pseudo-Anosov homeomorphism $\phi$, with respect to the translation surface $X(\pi, \theta)$, such that $D\phi = g_t$. In particular the dilatation of $\phi$ is $\theta$. Note that by construction $\phi$ fixes the zero on the left of the interval I and also the separatrix adjacent to this zero (namely the interval I).    Veech \cite{V} proved the following.

\begin{Thm}[Veech] \label{t:Veech}
Let $\gamma$ be a closed loop, beginning at the vertex corresponding to $\pi$, in a unlabeled Rauzy
diagram and $V$ be the associated transition matrix. If $V$ is primitive, then  let $\lambda$ be a positive eigenvector for the Perron eigenvalue $\theta$ of $V$ and $\tau$ be an eigenvector (with $\tau_{\pi_0^{-1}(1)}>0$) for the eigenvalue $\theta^{-1}$ of $V$. We have\\
(1) $\zeta = \lambda + i  \tau$ is a suspension datum for $T = (\pi, \lambda)$;\\
(2) The matrix $A = \begin{pmatrix} \theta^{-1} & 0 \\ 0 & \theta \end{pmatrix}$ is the derivative map of a pseudo-Anosov homeomorphism
$\phi$ on $X(\pi, \zeta)$;\\
(3) The dilatation of $\phi$ is $\theta$;\\
(4) Up to conjugation, all orientable pseudo-Anosov homeomorphisms fixing a separatrix are obtained by this construction.  
\end{Thm}

We note that  reversing the path $\gamma$, i.e. switching types 0 and 1 throughout, results in the inverse pseudo-Anosov homeomorphism. 

\subsubsection{Hyperelliptic Rauzy diagrams}
Our new families of examples of pseudo-Anosov maps with vanishing SAF-invariant are constructed using hyperelliptic diagrams.

Up to now, we have discussed labeled  IETs.   An unlabeled IET is one for which we retain only combinatorial data in the form of  a permutation of $\{1, . . . , d\}$.     Equivalent classes of unlabeled IETs  are obtained after identifying $(\pi_0,\pi_1)$ with $(\pi_0',\pi_1')$ if $\pi_1 \circ \pi_0^{-1}=\pi_1' \circ \pi_0'^{-1}$.  (See Viana \cite{Vi}, for a discussion of this, where the key term is ``monodromy".) From this, the labeled Rauzy diagram is a covering of the so-called unlabeled Rauzy diagram.   

An interval exchange transformation $T$ is called {\em hyperelliptic} if the corresponding permutation is such that 
$\pi_1 \circ \pi_0^{-1}(i)= d+1-i, \forall i = 1,...,d$.    A particular example of such a combinatorial datum is $\varepsilon_d := \begin{pmatrix} 1&2&\cdots&d\\d&d-1&\cdots & 1\end{pmatrix}$, with corresponding monodromy permutation $(d, d-1, \dots, 1)$.   A {\em hyperelliptic Rauzy diagram} is one that contains a combinatorial datum $\pi$ of a hyperelliptic IET.   Exactly when a Rauzy diagram is hyperelliptic, the labeled and unlabeled diagrams are isomorphic directed graphs.  See Figure~\ref{f:unlabHyper4} for the unlabeled Rauzy diagram with four subintervals.

\begin{figure}[h]
\scalebox{.4}{
\includegraphics{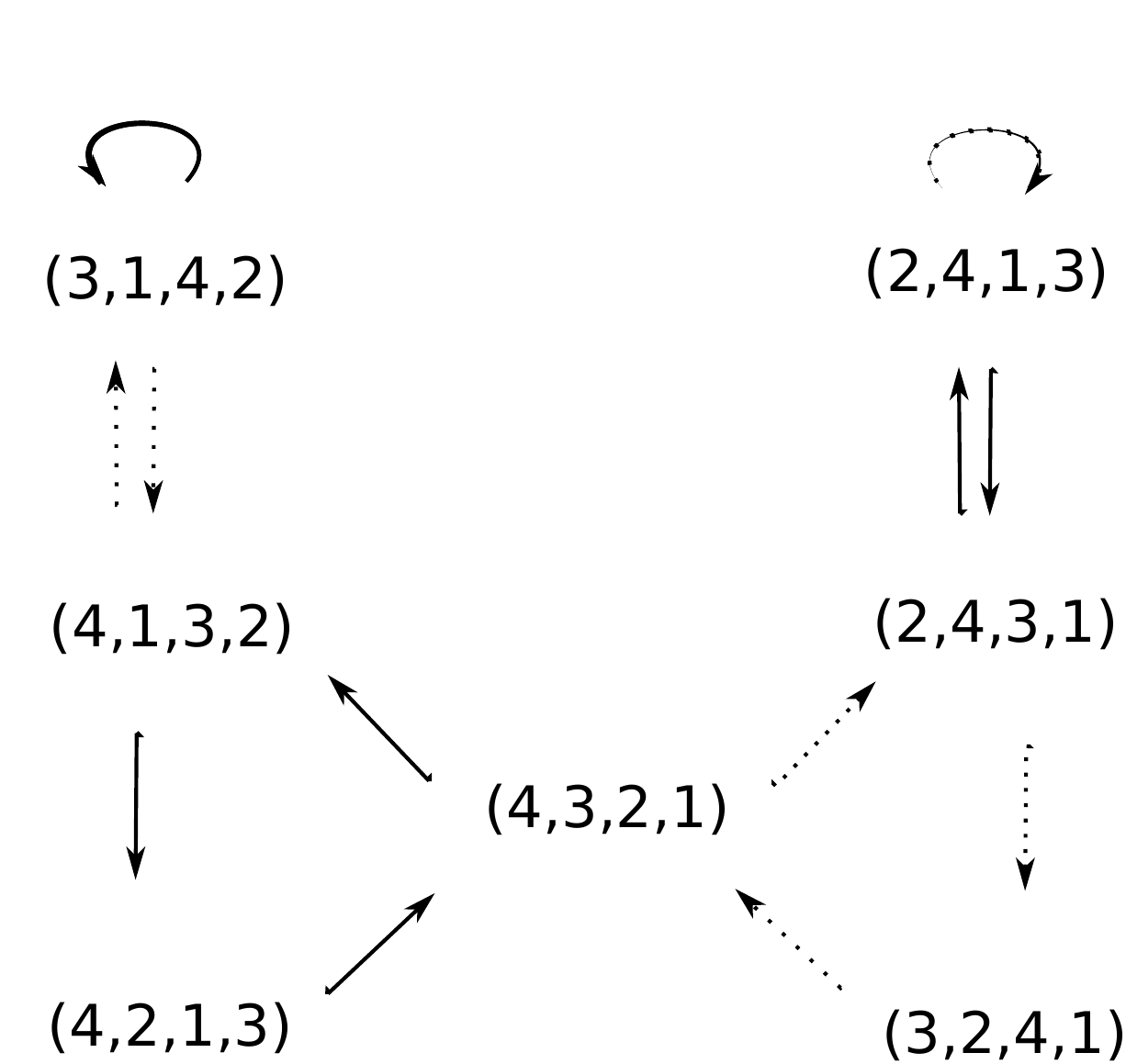}
}
\caption{Unlabeled hyperelliptic Rauzy diagram with 4 subintervals. Here and throughout, type 1 moves are shown by dotted lines, type 0 by solid.}
\label{f:unlabHyper4}
\end{figure}

  In our examples,  we  always choose the ``central"  vertex of the hyperelliptic diagram at hand to be the initial vertex of our path.  (The resulting pseudo-Anosov map is a conjugate of that given from taking any other initial vertex along the path.)  The following justifies that normalization, confer Figures~\ref{f:unlabHyper4}, \ref{f:VeechHep}, \ref{f:hypSevenGraph}.
  
  . 
  
\begin{Lem}\label{l:hypDiagram}   Suppose that $\gamma$ is a closed path in a hyperelliptic Rauzy diagram such that the corresponding transition matrix $V = V(\gamma)$ is primitive.   Then $\gamma$ must pass through the vertex corresponding to $\varepsilon_d$. 
\end{Lem}
\begin{proof} Since $\gamma$ is a closed path,  repeating it defines an infinite path in the Rauzy diagram.  By a result of Yoccoz \cite{Y}, which \cite{Vi} reproduces, every letter must hence be a winner on $\gamma$.    By the combinatorics of the induction,  the first letter on the top row of $\varepsilon_d$ is a winner exactly at the vertices forming a cycle of length $d-1$ in the hyperelliptic Rauzy diagram, similarly for the first letter on its bottom row.   Of course these cycles meet at the vertex corresponding to $\varepsilon_d$.   Therefore,  $\gamma$ must pass through this vertex. 
\end{proof}

\subsection{Components of strata and Rauzy classes}  In particular to allow experts to immediately understand the setting of our examples, we entitle certain subsections below with reference to particular components of strata of abelian differentials.   Here we briefly summarize the notation and related notions. 

Let $g \ge 2$ be the genus of the Riemann surface $X$, the non-zero abelian differentials on $X$ have zeros whose multiplicities sum to $2g-2$.   Let $\kappa$ be a partition of $2g-2$, the {\em stratum} $\mathscr H(\kappa)$ is the set (modulo the action of the mapping class group) of abelian differentials whose zeros have the multiplicities of $\kappa$.    

Computations by Veech and then Arnoux using Rauzy classes showed that in general strata have more than one connected component.   Kontsevich and Zorich   \cite{KZ} determined all possible components.   They showed that  any stratum has at most three components: there may be a hyperelliptic component where both $X$ is hyperelliptic and the hyperelliptic involution preserves $\omega$; and possibly two more components, differentiated by the parity of an appropriate notion of spin, these components are thus called ``even" and ``odd", correspondingly. One denotes the various components by $\mathscr H^{\text{hyp}}(\kappa)$, $\mathscr H^{\text{even}}(\kappa)$ and $\mathscr H^{\text{odd}}(\kappa)$.

   Our examples are in low genus, thus we recall only (part of) the second theorem of \cite{KZ}:  Each of $\mathscr H(2)$ and $\mathscr H(1,1)$ is connected (and coincides with its hyperelliptic component), while  each of $\mathscr H(4)$ and $\mathscr H(2,2)$ has two connected components:  a hyperelliptic component, and an odd spin component.   
   
    Each Rauzy class corresponds to a single component (see \cite{B} for details on this correspondence), and indeed one finds that the number of intervals $d$ is equal to $2 g + \sigma - 1$, where $\sigma$ equals the total number of zeros of the corresponding abelian differentials.  This accords with the fact that local coordinates on $\mathscr H(\kappa)$ are given by period coordinates, which one can view as the integration of $\omega$ over a basis of relative homology $H_1(X, \Sigma, \mathbb C)$, where $\Sigma$ is the set of zeros of $\omega$.  One can take the basis to be the union of an integral symplectic basis of absolute homology with a set of paths from a chosen zero to each of the other zeros.   
    
     The transition matrix $V = V(\gamma)$ for a closed path gives the action of the element of the mapping class group on relative homology.  In the pseudo-Anosov case, there is some power of the map that fixes all of $\Sigma$ and hence this power changes any path connecting zeros by an element of absolute homology.   On absolute homology,  the pseudo-Anosov (and perforce any of its powers) acts integrally symplectically, thus the action on relative homology of the power decomposes naturally into a block form with the block corresponding to pure relative homology being an identity.  Thus, the characteristic polynomial of this action is the product of a reciprocal degree $2g$ polynomial times a power of $(x-1)$.   Therefore, the action of the original pseudo-Anosov has a similar decomposition, as seen in our examples below.

\section{Characterization of vanishing SAF invariant,   implications}    We aim to prove that a pseudo-Anosov map has vanishing SAF-invariant exactly when an algebraic condition holds; we thus naturally first gather some algebraic results.

\subsection{Galois theory} 
We begin with a result using elementary Galois theory.  

\begin{Prop}\label{p:recipOrElse}  Suppose that $\alpha$ is a non-zero (irrational) algebraic number.   The minimal polynomial of $\alpha$ over $\mathbb Q$ is reciprocal  if and only if  $\mathbb Q(\alpha) \neq \mathbb Q(\alpha + \alpha^{-1})$.   
\end{Prop} 
\begin{proof} 
Let  $p(x) \in \mathbb Q[x]$  be the minimal polynomial of $\alpha$.     Let $q(x) \in \mathbb Q[x]$  be the minimal polynomial of $\alpha + \alpha^{-1}$.  Denote the degree of  $q(x)$ by $n$.  Since $\alpha$ satisfies $x^2 - (\alpha + \alpha^{-1})x + 1$ and of course $\mathbb Q(\alpha) \supseteq \mathbb Q(\alpha + \alpha^{-1})$, the degree of $p(x)$ is either $n$ or $2n$. \\

\noindent
$(\Leftarrow)$
 Set $\tilde q(x) = x^n q(x+ x^{-1})$.  Then $\tilde q(x) \in \mathbb Q[x]$ is monic of degree $2n$.  Of course,   $\tilde q(x)$ has $\alpha$  as a root.  Therefore,  $p(x)$ divides $\tilde q(x)$, and by the restrictions on the degree of $p(x)$,  either $p(x) =    \tilde q(x)$ or else $p(x)$ has degree $n$.   If $p(x)$ is not reciprocal, then it cannot equal $\tilde q(x)$, as this latter is clearly reciprocal; it then follows that  $n= [\mathbb Q(\alpha):\mathbb Q] = [\mathbb Q(\alpha + \alpha^{-1}):\mathbb Q]$, and thus $\mathbb Q(\alpha) = \mathbb Q(\alpha + \alpha^{-1})$.\\

\noindent
$(\Rightarrow)$  Suppose now that $\mathbb Q(\alpha) = \mathbb Q(\alpha + \alpha^{-1})$.  
Recall that any root of $q(x)$ is the image of $\alpha+\alpha^{-1}$ under some field embedding (fixing $\mathbb Q$),  $\mathbb Q(\alpha + \alpha^{-1})\hookrightarrow \mathbb C$.  Since $\mathbb Q(\alpha) = \mathbb Q(\alpha + \alpha^{-1})$,   each such field embedding sends $\alpha$ to some root of $p(x)$.  This field equality also implies that $\deg p(x) = n$,  and thus we conclude that the roots of $q(x)$ are all contained in the set of values of the form $\beta+\beta^{-1}$ with $\beta$ a root of $p(x)$.      However, under the further supposition that $p(x)$ is reciprocal (which implies that $n$ is even, see Lemma~\ref{l:reciprocalIrredIsEvenDeg}),  there are only $n/2$ {\em distinct} values in the set of the $\beta+\beta^{-1}$.    Hence, the degree of $q(x)$ must in fact be at most $n/2$, and we have reached a contradiction.

\end{proof} 

For the sake of completeness, we include the following well-known result. 
\begin{Lem}\label{l:reciprocalIrredIsEvenDeg}  Suppose that $p(x) \in \mathbb Z[x]$ is reciprocal and of odd degree (greater than one).   Then $p(x)$ is reducible.   
\end{Lem}
\begin{proof} Given the $p(x)$ is reciprocal, whenever $\alpha$ is a root pf $p(x)$, so is $\alpha^{-1}$.   Thus,  the roots of $p(x)$ are paired together by $x \mapsto 1/x$.  This accounts for an even number of roots, except for fixed points of this map.  Since $p(x)$ has an odd number of roots, we conclude that at least one of the fixed points,  $x=\pm 1$, is a root of $p(x)$. It follows that $p(x)$ is reducible. 
\end{proof}

We draw some immediate conclusions from Proposition~\ref{p:recipOrElse}.  Let us introduce a non-standard definition:  Call $\mathbb Q(\alpha + \alpha^{-1})$ the {\em trace field} of the algebraic number $\alpha$.     Recall that the (algebraic) norm of an algebraic number is the product of all of its conjugates over $\mathbb Q$.  

\begin{Cor}\label{c:normOneQuadratic}  If $\alpha$  is of norm one with quadratic trace field, then $\mathbb Q(\alpha) \neq \mathbb Q(\alpha + \alpha^{-1})$.
\end{Cor}
\begin{proof}  Field equality would imply that $\alpha$ is quadratic,  and hence with minimal polynomial of the form $p(x) = x^2 + nx + 1$ for some $n \in \mathbb Z$.  But $p(x)$ is reciprocal of even degree and hence field equality cannot hold.
 
\end{proof}

\begin{Cor}\label{c:PisotImpNonRecip&FldSafZero}  If $\alpha$ is a non-quadratic Pisot number, then $\alpha$ is non-reciprocal.  Moreover,   
$\mathbb Q(\alpha)= \mathbb Q(\alpha + \alpha^{-1})$.
\end{Cor}
\begin{proof} 
Since the minimal polynomial $p(x)$ of $\alpha$ has degree greater than two, it has a root $\beta \neq \alpha^{-1}$ with 
$||\beta || < 1$,  therefore $||\beta^{-1}||>1$ but since $p(x)$ has only $\lambda$ as a root that has norm greater than one, we conclude that 
$p(x)$ is not a reciprocal polynomial.  Thus, we can invoke Proposition~\ref{p:recipOrElse} to find that also the second statement holds.
 
\end{proof}

Motivated by this last result,  we now show that every Pisot unit is bi-Perron (which presumably is well-known).  
\begin{Lem}\label{l:PisotImpliesBiPerron}  If $\alpha$ is a non-quadratic Pisot unit, then 
$\alpha$ is bi-Perron.
\end{Lem}
\begin{proof} 
Let $\alpha = \alpha_1, \dots,  \alpha_n$ be the roots of the minimal polynomial of $\alpha$.  Then we have that 
$||\alpha_1 \cdots \alpha_n|| =1$,  $||\alpha_1||>1$ and for each $j> 1$, $||\alpha_j ||<1$.   Therefore for each $i>1$ we have 
\[||\alpha_i|| = \dfrac{1}{||\alpha_1||}\, \dfrac{1}{\prod_{j>1, j \neq i}\; ||\alpha_j||}\,\]
and thus $||\alpha_i||> 1/||\alpha_1||$ and the result holds.   
\end{proof}

\bigskip 

\begin{Lem}\label{l:cubicPisotUnits}    The set of cubic bi-Perron units is exactly the set of cubic Pisot units.  
\end{Lem}
\begin{proof}    Suppose that $x^3 + a x^2 + b x + c \in \mathbb Z[x]$ with $c = \pm 1$ has non-zero roots $\alpha_1, \alpha_2, \alpha_3$.   Then $|| \alpha_2 \alpha_3 || = 1/||\alpha_1||$, and hence  $\alpha_2$ satisfies $|| \alpha_2 ||> 1/||\alpha_1 ||$ if and only if $|| \alpha_3 ||<1$ and similarly with the roles of $\alpha_2, \alpha_3$ exchanged.    Thus, if $\alpha_1$ is indeed cubic, then its two conjugates have norm greater than its inverse if and only if they both lie in the unit disk.  Since $-c = \alpha_1 \alpha_2 \alpha_3$,  it also follows that in this setting  $\alpha_1$ is of norm greater than one.    Thus, we find that every cubic bi-Perron number is indeed a Pisot unit.  Of course, the previous result gives the other inclusion.
\end{proof}    
   
\begin{Rmk}   On the other hand,  not every non-reciprocal Perron unit is a Pisot number, as already 
$f(x) =   x^4-4x^3+3x+1$ shows.
\end{Rmk}

\bigskip

We verify a property required for a certain construction of pseudo-Anosov elements.   
\begin{Lem}\label{l:nonRecipNotCyclo}  If $\alpha$ is a non-reciprocal Perron number, let $\tilde q(x)$ be the product of the minimal polynomial of $\alpha$ with the minimal polynomial of $\alpha^{-1}$.   Then $\tilde q(x) = f(x^k)$ with $f(x) \in \mathbb Z[x]$ and $k \in \mathbb N$ implies $k=1$.  
\end{Lem}
\begin{proof} Suppose $\tilde q(x) = f(x^k)$,  then certainly for any zero $\beta$ of $f(x)$ and every $k^{\text{th}}$-root $\gamma$ of $\beta$, thus satisfying $\gamma^k = \beta$, we have $\tilde q(\gamma) = 0$.  Hence the zeros of $\tilde q(x)$ form the full set of the $k^{\text{th}}$-roots of the various zeros of $f(x)$.    But, for $\beta$ fixed, all of its  $k^{\text{th}}$-roots share the same complex norm.  Therefore, we can partition the set of roots of $\tilde q(x)$ into subsets of cardinality $k$ with all elements of the subset sharing the same complex norm.    However,  since $\alpha$ is bi-Perron,  there is no other root of 
 $\tilde q(x)$ that has the same complex  norm as does $\alpha$.    We conclude that $k=1$ and of course $f(x) = \tilde q(x)$. 
\end{proof}

Similarly, we have the following.
\begin{Lem}\label{l:recipNotCyclo}  If $\alpha$ is a  reciprocal bi-Perron number and $p(x)$ its minimal polynomial, then $p(x) = f(x^k)$ with $f(x) \in \mathbb Z[x]$ and $k \in \mathbb N$ implies $k=1$.  
\end{Lem}
\begin{proof}  Here also, the polynomial in question has $\alpha$ as its only root that is of complex norm $||\alpha||$.   Thus, the argument used to prove the previous Lemma applies. 
\end{proof}

\subsection{Proof of Theorem~\ref{t:characterization}:  Characterizing SAF-zero pseudo-Anosov maps}

\bigskip

\begin{proof} Suppose that  $\phi$ is a pseudo-Anosov map with dilatation $\lambda$.  By the results of Calta-Smillie reviewed in Subsection~\ref{ss:especiallyCaltaSmillie}, we can assume that $\phi$ is an affine diffeomorphism on $(X, \omega)$ with matrix part being hyperbolic in $k = \mathbb Q(\lambda + \lambda^{-1})$ and that $\phi$ has vanishing SAF-invariant if and only if its stable direction has slope in $k$.

The fixed points under the M\"obius action on $\mathbb R \cup \{\infty\}$ of $M = \begin{pmatrix} a&b\\c&d\end{pmatrix}$  hyperbolic  in $\text{SL}_2(\mathbb R)$ are $(a-d)\pm \sqrt{(a+d)^2 - 4}/(2 c)$.    Due to the projective nature of this action,  the corresponding eigenvectors of $M$ have slopes that are the inverses of these fixed points.  Thus, these eigenvectors are of slope in $k$ exactly when $(a+d)^2 - 4$ is a square in $k$.   When $M$ is the matrix part of the affine diffeomorphism $\phi$ on $(X, \omega)$,  the eigenvectors of $M$ give the direction of the stable and unstable foliations for $\phi$ on $(X, \omega)$, and the trace of $M$ equals $\lambda + \lambda^{-1}$.    Thus, these foliations have directions in the trace field exactly when $(\lambda + \lambda^{-1})^2 - 4$ is a square in $k$.   That is,  $\phi$ has vanishing SAF-invariant if and only if $(\lambda + \lambda^{-1})^2 - 4$ is a square in $k$.

On the other hand,  it is obvious that $\mathbb Q(\lambda) \supset k$ and that $\lambda$ is a zero of $(x-\lambda)(x-\lambda^{-1}) = x^2 - (\lambda + \lambda^{-1}) x + 1 \in k[x]$.   Hence $\mathbb Q(\lambda) =  k$ if and only if the discriminant of $x^2 - (\lambda + \lambda^{-1}) x + 1$ is a square in $k$, and otherwise there is a proper containment with field extension degree $[\mathbb Q(\lambda):k] = 2$.    However, the discriminant of $x^2 - (\lambda + \lambda^{-1}) x + 1$ is $(\lambda + \lambda^{-1})^2 - 4$.  Thus,  we find that $\phi$ has vanishing SAF-invariant if and only if $\mathbb Q(\lambda) =   \mathbb Q(\lambda + \lambda^{-1})$. 

Our result now follows from Proposition~\ref{p:recipOrElse}.
\end{proof}

\bigskip 

\begin{Rmk}   Since the stable and unstable foliations of the pseudo-Anosov map correspond to the fixed points of the linear part, it follows from the above proof that either both are of vanishing SAF-invariant, or else neither is. 
\end{Rmk}

\subsection{Some implications}

Recall that if $\lambda$ is the dilatation of a (orientable) pseudo-Ansov map $\phi$, then we call $\mathbb Q(\lambda + \lambda^{-1})$ the {\em trace field} of $\phi$.
\begin{Cor}\label{c:quadratic}  If an orientable pseudo-Anosov map $\phi$ has quadratic trace field, then $\phi$ has non-vanishing \textnormal{SAF}-invariant. 
\end{Cor}

\begin{proof}   The dilatation of $\phi$ is a unit, and hence has norm $\pm 1$.  In the first case, Corollary~\ref{c:normOneQuadratic}  applies.  In the second,  the minimal polynomial (being monic) cannot be reciprocal. 
\end{proof}

\bigskip 

\begin{Rmk}  Recall that Kenyon-Smillie \cite{KS} showed that if $(X, \omega)$ supports an affine pseudo-Anosov map, then the trace field of the map is the trace field of $(X, \omega)$.   We can thus compare Corollary~\ref{c:quadratic} with  McMullen's  Theorem~A.1 of the appendix in \cite{McFlux}.  In our language, McMullen shows that under the hypothesis that the Veech group of $(X,\omega)$ is a lattice (which certainly implies the existence of affine pseudo-Anosov maps), the trace field of $(X, \omega)$ being quadratic implies that the only directions of flow with vanishing SAF-invariant are those for which the flow is periodic. (By Veech's dichotomy \cite{V2}, these are the  directions in which $(X, \omega)$ decomposes into cylinders).  McMullen's result thus certainly excludes the stable foliation of pseudo-Anosov maps from having vanishing SAF-invariant.     That is,   the hypothesis of lattice Veech group allows one also  to rule out (from having vanishing SAF-invariant) those non-periodic directions which are not the stable foliations of pseudo-Anosov maps. 
\end{Rmk}

\bigskip 

\begin{Rmk}   We point out that  if a  pseudo-Anosov map $\phi$ is of vanishing SAF-invariant and its dilatation is not totally real, then its trace field is also not totally real.   This holds, as vanishing SAF-invariant implies equality of the trace field with the field generated over $\mathbb Q$ by the dilatation. 
This can be applied to allow a minor simplification in the existence arguments of \cite{HL}. 

\end{Rmk}

 \bigskip 
         
\subsection{Every bi-Perron unit has its minimal polynomial dividing the characteristic polynomial of some pseudo-Anosov's homological action}\label{ss:biPerronHomol}
 \bigskip

\begin{Thm}\label{t:realization}   Suppose that $\alpha$ is a bi-Perron unit.   Then the minimal polynomial of $\alpha$ divides the characteristic polynomial of the action on first integral homology  induced by some pseudo-Anosov map. 
\end{Thm}

 \begin{proof} If $\alpha$ is a  bi-Perron unit  whose minimal polynomial $p(x)$ is reciprocal (and hence of even) degree say $2 g$, then $p(x)$ is obviously (symplectically) irreducible.    That $p(x)$  is not cyclotomic   is clear.  That $p(x) = f(x^k)$ is only trivially possible is shown in Lemma~\ref{l:recipNotCyclo}.   Thus, the hypotheses are all satisfied for the Margalit-Spallone construction of \cite{MS} to give an explicit pseudo-Anosov element, (indeed a full coset of the Torelli group) in the mapping class group of the genus $g$  surface,  whose induced action on homology has characteristic polynomial $p(x)$.

If $\alpha$ is a bi-Perron unit of   degree $g$  whose minimal polynomial $p(x)$ is not reciprocal, let  $\hat p(x)$ be the minimal polynomial of $\alpha^{-1}$.  And once again let $q(x)$ be the  minimal polynomial of $\alpha + \alpha^{-1}$, which by Theorem~\ref{t:characterization} is also of degree $g$.  Let  $\tilde q (x) = x^g q(x+x^{-1})$.  Since both  $\alpha, \alpha^{-1}$ are roots of $\tilde q (x)$, degree considerations give that   $\tilde q(x) = p(x) \hat p(x)$.

Lemma~\ref{l:nonRecipNotCyclo} shows that $\tilde q(x)$ is not equal to any non-trivial $f(x^k)$.  
That $\tilde q(x)$ has no cyclotomic roots is clear, as its only roots are those of  $p(x), \hat p(x)$ and each of these is an irreducible polynomial with a root that is of absolute value greater than one.     Again, the hypotheses are all satisfied for the Margalit-Spallone construction, so that there exist pseudo-Anosov homeomorphisms whose induced action on homology is of characteristic polynomial $\tilde q (x)$.  

\end{proof}

\begin{Rmk}    If  any of the pseudo-Anosov maps guaranteed by \cite{MS} is orientable,  then its  dilatation is an eigenvalue of the action on homology.  But, the dilatation  must then be $\alpha$, and we have realized $\alpha$ as a dilatation.  

As recalled in \cite{LT}, the dilatation of  a non-orientable pseudo-Anosov homeomorphism cannot be an eigenvalue for the induced action on homology.  Thus, if  none of the pseudo-Anosov maps guaranteed by \cite{MS} is orientable,    even after applying the  standard double cover construction (see say the text \cite{FM})  we can say little more than stated in Theorem~\ref{t:realization}.   
\end{Rmk}

\bigskip 
\bigskip 
\subsection{A problem of Birman {\em et al.}}\label{ss:BirmanQ}
Birman, Brinkmann and Kawamuro   \cite{BBK} associate to a   pseudo-Anosov map $\phi$ of dilatation  $\lambda$  a symplectic polynomial $s(x)$ that has $\lambda$ as its largest real root.  They write ``its relationship to the minimum polynomial of $\lambda$ is not completely clear at this writing."   We give an explanation in the setting that $\phi$ is orientable  (and defined on a surface without punctures).

\begin{Thm}\label{t:birmanEtAl}    Suppose that $\phi$ is an orientable pseudo-Anosov map on a surface of genus $g$.
Let $s(x)$ be the polynomial associated to $\phi$ in \cite{BBK}.  Then $s(x)$ is reducible if and only if either $\phi$ has vanishing \textnormal{SAF}-invariant or has trace field of degree strictly less than $g$. 
\end{Thm}

\bigskip 

\begin{proof}  Let $\lambda$ be the dilatation of $\phi$ and $p(x)$ be the minimal polynomial of $\lambda$.   Since $s(x)\in \mathbb Z[x]$ is monic and has $\lambda$ as a root, of course $p(x)$ divides $s(x)$.    As well, since $s(x)$ is a reciprocal polynomial,  whenever some $\alpha$ is a root of $s(x)$ so also is $\alpha^{-1}$ a root.

If $s(x)$ is irreducible then it equals $p(x)$.  Thus,  $p(x)$ is in particular reciprocal.   Therefore, by Theorem~\ref{t:characterization} the SAF-invariant of $\phi$  does not vanish.  

 Suppose now that $s(x)$ is reducible but symplectically irreducible.  Were $p(x)$ reciprocal, then there would exist some other factor of $s(x)$, but this factor would perforce be reciprocal.  This contraction shows that in this case $p(x)$ is not a reciprocal polynomial.  In particular,  the minimal polynomial $\hat p(x)$ of  $\lambda^{-1}$ is distinct from $p(x)$.   But since $\lambda$ is a root of $s(x)$, so is $\lambda^{-1}$ and hence $\hat p(x)$ also divides $s(x)$.   That is $\tilde q(x) = p(x) \hat p(x)$ divides $s(x)$.   The existence of any further factor of $s(x)$ would lead to a contradiction of the symplectic irreducibility of $s(x)$.   That is,  whenever  $s(x)$ is reducible but symplectically irreducible it is exactly the product $\tilde q(x) = p(x) \hat p(x)$ and $p(x)$ is not reciprocal.   By Theorem~\ref{t:characterization} the SAF-invariant of $\phi$ vanishes.  

Finally, suppose that $s(x)$ is  symplectically reducible.  We have that either $p(x)$ is reciprocal or that $\tilde q(x) = p(x) \hat p(x)$ divides $s(x)$.   In either case, there is some other reciprocal factor of $s(x)$.  Thus the degree of $p(x)$ or $\tilde q(x)$ is correspondingly of degree less than $2g$ and as the trace field $\mathbb Q(\lambda + \lambda^{-1})$ has dimension over $\mathbb Q$ equal to one-half of the  degree of  $p(x)$ or $\tilde q(x)$ in these respective cases, we indeed find that the trace field of $\phi$ has degree strictly less than $g$.
\end{proof}

\begin{Rmk}
In particular,  Example~5.2 of \cite{BBK} shows that the monodromy of the hyperbolic knot $8_9$ leads to an orientable  pseudo-Anosov map with  $s(x) = (x^3-2x^2 + x- 1)(x^3 - x^2 + 2x - 1)$.     Here the dilatation $\lambda$ is the real root of $x^3 - x^2 + 2x - 1$,  the second factor is the minimal polynomial of $1/\lambda$.  Using its minimal polynomial, one easily shows that $\lambda$ equals $-(\lambda + \lambda^{-1})^2 + 3 (\lambda + \lambda^{-1}) -1$,  
implying that indeed $\mathbb Q(\lambda)= \mathbb Q(\lambda + \lambda^{-1})$. 
\end{Rmk}

\section{Spinning about small loops}   
 
\subsection{Rediscovering the Arnoux-Rauzy family of  $\mathscr{H}^{\text{odd}}(2, 2)$}  Mimicking the construction of \cite{AY},  Arnoux and Rauzy \cite{AR} constructed an infinite family of IETs, the first two  of which  Lowenstein,  Poggiaspalla, and  Vivaldi \cite{LPV, LPV2}  studied in detail, as these lead to SAF-zero pseudo-Anosov maps.   Indeed, by making an appropriate adjustment,  Lowenstein {\em et al.} renormalized these first two IETs in such a way that each was periodic under Rauzy induction.   Each corresponds to a cycle passing through the same 29 vertices in the 294-vertex Rauzy class of 7-interval IETs, and under the Veech construction leads to a pseudo-Anosov homeomorphism.  The dilatations of these are the largest root of $x^3-7x^2+5x-1=0$ and $x^3-10x^2+6x-1=0$ respectively.      

Presumably,  Lowenstein {\em et al.}  intend that one follow their recipe for constructing pseudo-Anosov homeomorphisms for the remainder of the Arnoux-Rauzy family.   This seemed somewhat daunting to us.   However,  we found that one can succeed by adjusting the cycle  given by the first Arnoux-Rauzy IET  by spinning about certain small cycles.   Since the Arnoux-Yoccoz pseudo-Anosov homeomorphism in genus 3 corresponds to an abelian differential in  $\mathscr H^{\text{odd}}(2,2)$ (for this and much more see \cite{HLM}), all of these examples (since they arise from the same Rauzy class) are in this same connected component. 

More precisely, for each $k \geq1$, the path $\rho_k=00001010(111111)^{k-1}1101(00)^{k-1}010100111$, starting from the permutation $(7354621)$, gives these maps. (Here and throughout,  exponents  as in the expression for $\rho_k$ indicate repeated concatenation of the correspondingly grouped symbols.)  One then finds that the characteristic polynomial of the induced transition matrix for $\gamma_k$ is $p_k(x)=(x^3-(3k+4)x^2+(k+4)x-1)(x^3-(k+4)x^2+(3k+4)x-1)(x-1)$.  To verify this, break up $\rho_k$ into five paths corresponding to $00001010, (111111)^k, 1101, (00)^k, 010100111$,  and compute their transition matrices. From this,  one easily shows that the associated matrix for $\gamma_k$ is the matrix

 \[V_k=\begin{pmatrix}
 2 & 2 & 2 & 2 & 4 & 3k+2 & 3k-1\\
 0 & 2 & 2 & 1 & 0 & 0 & 0\\
 0 & k-1 & k & 0 & 0 & 0 & 0\\
 1 & 2 & 2 & 2 & 0 & 0 & 0\\
 0 & 0 & 0 & 0 & 2 & k+1 & k\\
 0 & 0 & 0 & 0 & 1 & 2k & 2k-2\\
 1 & 1 & 1 & 1 & 1 & 1 & 1\\ 
 \end{pmatrix}\]
whose characteristic polynomial is $p_k(x)$.

\begin{Rmk}   Erwan Lanneau has informed us that (in unpublished work) he also found this family,  in virtually the same manner.    
\end{Rmk}

\subsection{ Two known examples in  $\mathscr{H}^{\text{hyp}}(4)$}  Veech \cite{V2} constructed an infinite family of translation surfaces with Veech groups that are lattices in $\text{SL}_2(\mathbb R)$.   For each $n \ge 5$,  his construction is to identify, by translation, parallel sides of a  regular $n$-gon and its mirror image.  In the case of $n=7$,  one finds a genus 3 surface with exactly one singularity of cone angle $8\pi$.   Veech shows that the Veech group here is generated by $S=\begin{pmatrix} \cos(\pi/7) &  -\sin(\pi/7) \\ \sin(\pi/7)& \cos(\pi/7)\end{pmatrix}$ and $T=\begin{pmatrix} 1 &  2\cot(\pi/7) \\ 0& 1\end{pmatrix}$.    In \cite{AS}, it is pointed out that results on (Rosen) continued fractions of Rosen and Towse \cite{RT} imply that on this surface there is a SAF-zero pseudo-Anosov; indeed this is the map, say $\psi$, of linear part $D\psi = STS^{-1}T$.    Explicitly taking a transversal to the flow in the expanding direction for $\psi$, and following Rauzy induction on the corresponding IET,   we found that $\psi$ results from the loop displayed in Figure~\ref{f:VeechHep}.

 \begin{figure}[h]
\scalebox{.6}{
\includegraphics{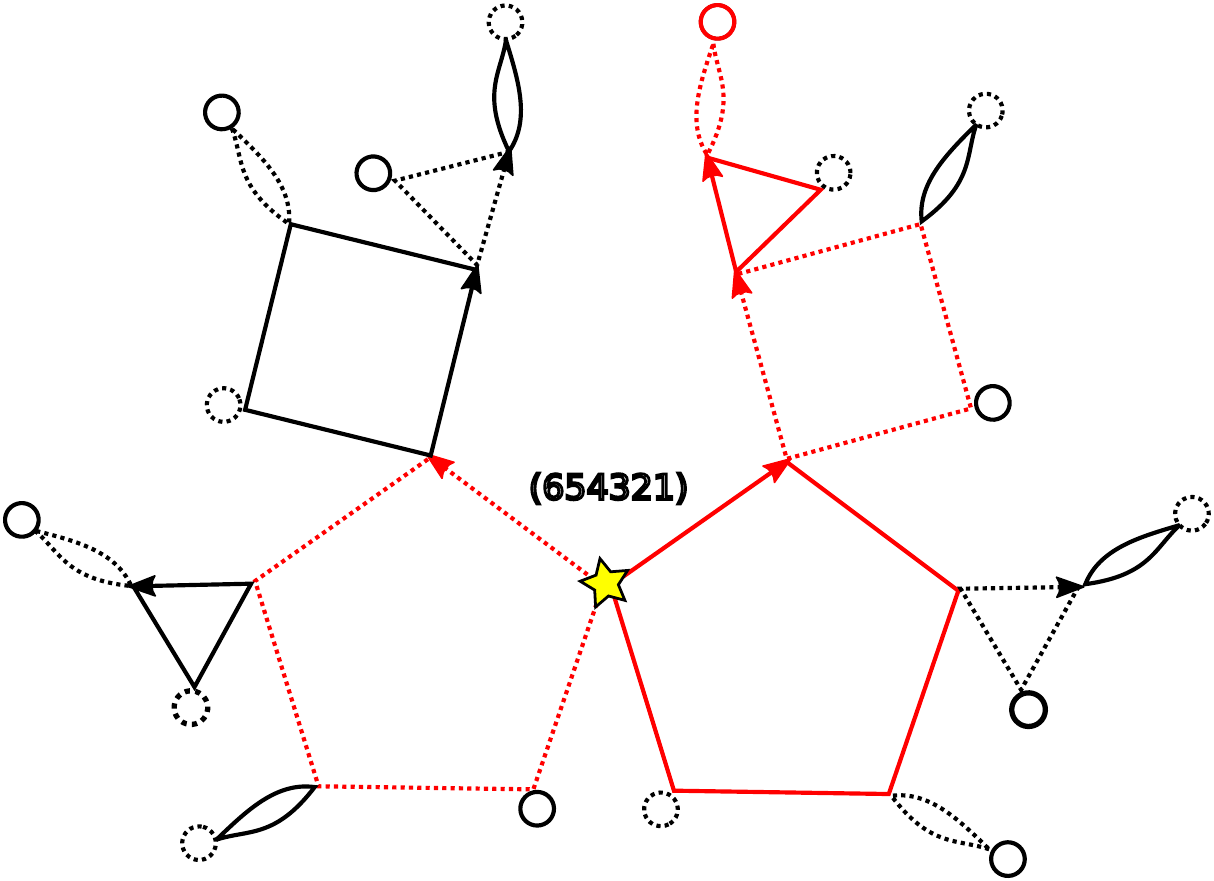}
}
\caption{Red loop representing $\psi$,  a pseudo-Anosov on Veech's double heptagon surface.}
\label{f:VeechHep}
\end{figure}
The primitive matrix associated with this loop has characteristic polynomial $(x^3-6x^2+5x-1)(x^3-5x^2+6x-1)$, verifying that the SAF invariant vanishes.

\bigskip 
Lanneau's example, given in \cite{McCasc},  has as its dilatation  the largest root of $x^3-8x^2+6x-1$.  We noticed that both $\psi$ and this example correspond to paths  passing through  the same 15 vertices of  the hyperelliptic Rauzy graph of 6-interval IETs.   These paths only differ in that  Lanneau's has added spins (indeed, the ``top right" 1-loop is repeated four times).

 \begin{figure}[h]
\scalebox{.4}{
\includegraphics{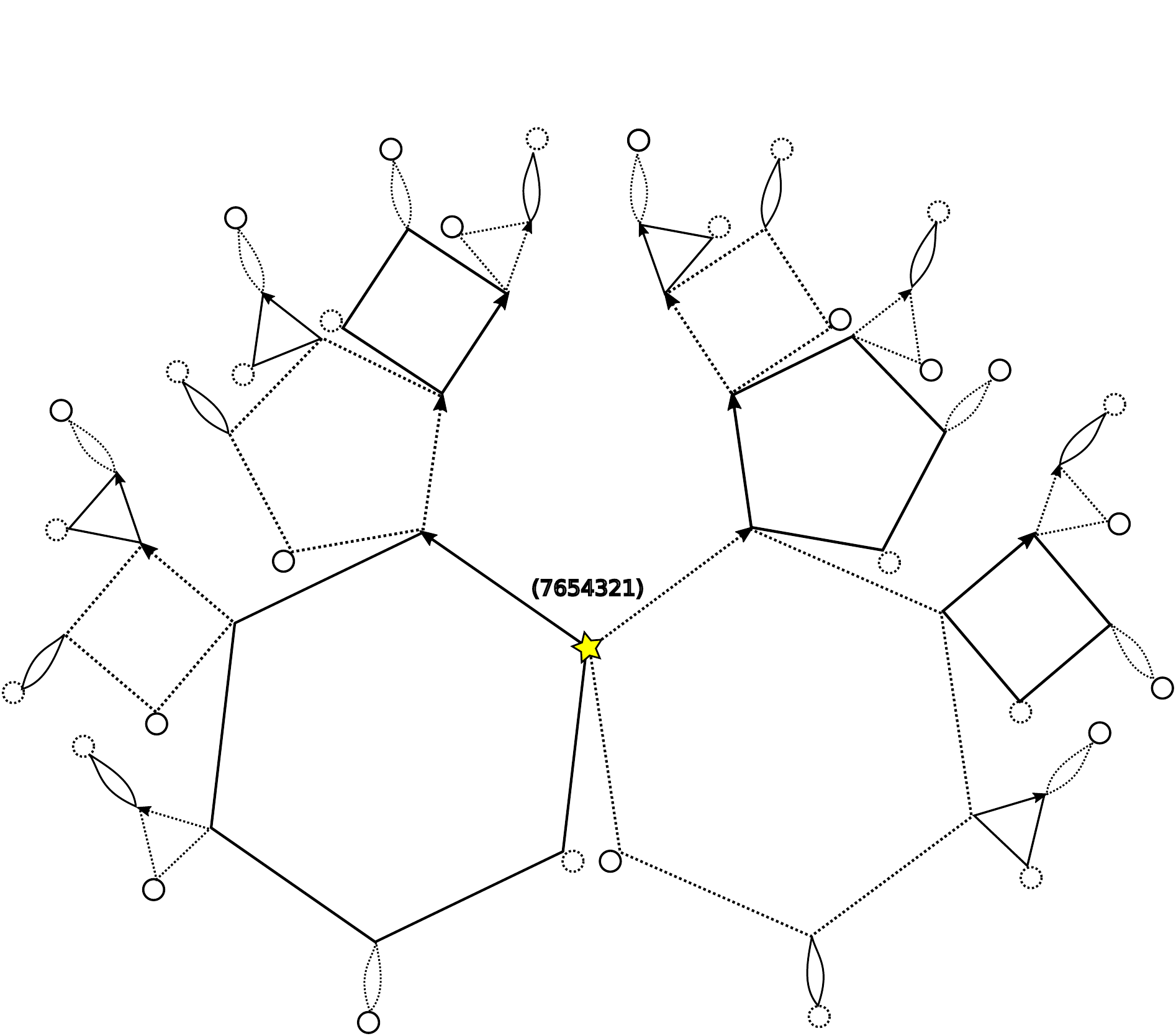}
}
\caption{Hyperelliptic Rauzy diagram with 7 sub-intervals.}
\label{f:hypSevenGraph}
\end{figure}
\subsection{New families of pseudo-Anosov maps in $\mathscr{H}^{\text{hyp}}(2,2)$} 
Motivated by the previous examples, we sought an infinite family of pseudo-Anosov with dilatation being the largest root of $x^3-(2k+4)x^2+(k+4)x-1$.  We found such a family, but rather by taking certain paths in the hyperelliptic Rauzy graph of 7-intervals IETs.   This graph is shown in Figure~\ref{f:hypSevenGraph}.   We find in fact four distinct families, and thus new examples of pseudo-Anosov maps with vanishing SAF-invariant.   Naturally enough, we describe the paths as starting at the vertex of  $\pi=(7654321)$.

\subsubsection{Closed loops $\alpha_k$}
For   $k\geq2$, let  $\alpha_k$ be given by $10101(0^{k-1})10011100001111100000(1^{k-1})0$, see Figure~\ref{f:fourLoops}.  We obtain the following transition matrix 

\[V_k=\begin{pmatrix}
 2 & 2 & 2 & 2 & 2 & k+1 & k\\
 0 & 2 & 2 & 2 & 2 & k & k-1\\
 0 & 0 & 2 & 2 & 1 & 0 & 0\\
 0 & 0 & k-1 & k & 0 & 0 & 0\\
 0 & 1 & 2 & 2 & 2 & 0 & 0\\
 1 & 2 & 2 & 2 & 2 & 2k & 2k-2\\
 1 & 1 & 1 & 1 & 1 & 1 & 1\\ 
 \end{pmatrix}\,.\]
This has  characteristic polynomial $(x^3-(2k+4)x^2+(k+4)x-1)(x^3-(k+4)x^2+(2k+4)x-1)(x-1)$. Hence, the corresponding pseudo-Anosov map $\phi_{\alpha_k}$ has vanishing SAF-invariant.

\subsubsection{Closed loops $\beta_k$}

Let $\beta_k:11010101(0^{k-1})1000111100000(1^{k-1})0$ for $k\geq2$, see Figure~\ref{f:fourLoops}. We obtain the  transition matrix
\[V_k=\begin{pmatrix}
 2 & 2 & 2 & 2 & 2 & k+1 & k\\
 0 & 2 & 2 & 2 & 1 & 0 & 0\\
 0 & 2k-2 & 2k & 1 & 0 & 0 & 0\\
 0 & k & k+1 & 2 & 0 & 0 & 0\\
 1 & 2 & 2 & 2 & 2 & 0 & 0\\
 0 & 0 & 0 & 0 & 0 & k & k-1\\
 1 & 1 & 1 & 1 & 1 & 1 & 1\\ 
 \end{pmatrix}\,\]
 whose characteristic polynomial is $(x^3-(2k+4)x^2+(k+4)x-1)(x^3-(k+4)x^2+(2k+4)x-1)(x-1)$. Hence, the corresponding pseudo-Anosov map $\phi_{\beta_k}$ has vanishing SAF-invariant.

\begin{figure}[h]
\noindent
\begin{tabular}{cc}

 \scalebox{.4}{
  \includegraphics{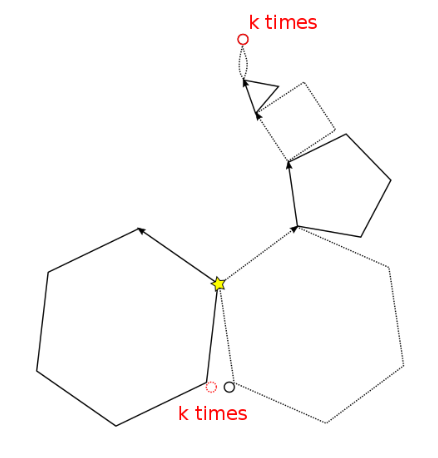}
  }
& 
\scalebox{.4}{
  \includegraphics{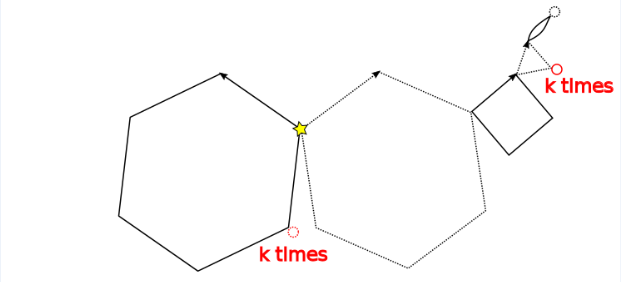}
  } 
 \\
  \scalebox{.4}{
  \includegraphics{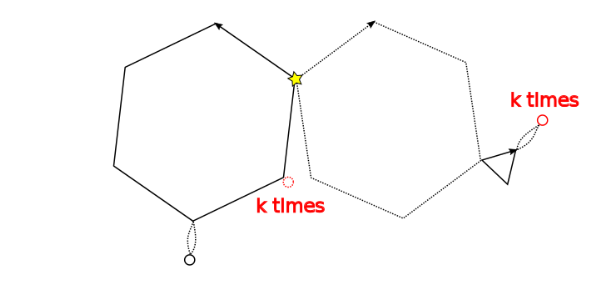}
  }
& 
\scalebox{.4}{
  \includegraphics{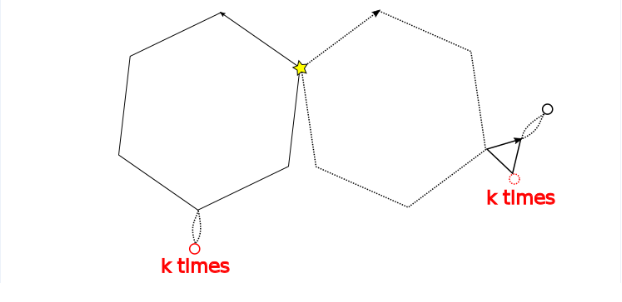}
  } 
  \end{tabular}
\caption{Clockwise from top left: The paths for $\alpha_{k+1}, \beta_{k+1}, \gamma_{k+1}, \delta_{k+1}\,$. }
\label{f:fourLoops}
\end{figure}

\subsubsection{Closed loops $\gamma_k$}
Let $\gamma_k:11101010(1)^{k-1}011100001(0)^{k-1}100$ for $k\geq2$, see Figure~\ref{f:fourLoops}. We obtain the transition matrix  
\[V_k=\begin{pmatrix}
 2 & 2 & 2 & 2 & 4 & 3k+2 & 3k-1\\
 0 & 2 & 2 & 1 & 0 & 0 & 0\\
 0 & k-1 & k & 0 & 0 & 0 & 0\\
 1 & 2 & 2 & 2 & 0 & 0 & 0\\
 0 & 0 & 0 & 0 & 2 & k+1 & k\\
 0 & 0 & 0 & 0 & 1 & 2k & 2k-2\\
 1 & 1 & 1 & 1 & 1 & 1 & 1\\ 
 \end{pmatrix}\,\]
whose characteristic polynomial is $(x^3-(2k+4)x^2+(k+4)x-1)(x^3-(k+4)x^2+(2k+4)x-1)(x-1)$. Hence, the corresponding pseudo-Anosov map $\phi_{\gamma_k}$ has vanishing SAF-invariant.

\subsubsection{Closed loops $\delta_k$}

Let $\delta_k:11101(0^{k-1})10011100001010(1^{k-1})0$ for $k\geq2$, see Figure~\ref{f:fourLoops}. We obtain the   transition matrix   
\[V_k=\begin{pmatrix}
 2 & 2 & 2 & 2 & k+2 & k+3 & 2\\
 0 & 2 & k+1 & k & 0 & 0 & 0\\
 0 & 1 & 2k & 2k-2 & 0 & 0 & 0\\
 1 & 2 & 2 & 2 & 0 & 0 & 0\\
 0 & 0 & 0 & 0 & 2 & 2 & 1\\
 0 & 0 & 0 & 0 & k-1 & k & 0\\
 1 & 1 & 1 & 1 & 1 & 1 & 1\\ 
 \end{pmatrix}\,\]
whose characteristic polynomial is $(x^3-(2k+4)x^2+(k+4)x-1)(x^3-(k+4)x^2+(2k+4)x-1)(x-1)$. Hence, the corresponding pseudo-Anosov map $\phi_{\delta_k}$ has vanishing SAF-invariant.

\subsection{Examples in other strata} The similarities in patterns of the paths for the Veech heptagon and Lanneau's pseudo-Anosov maps  in $\mathscr{H}^{\text{hyp}}(4)$ and those of our examples in $\mathscr{H}^{\text{hyp}}(2,2)$, motivated us to  investigate  similar patterns in the Rauzy hyperelliptic diagrams for 8 and 9 subintervals.   We discovered several isolated examples of SAF-zero maps here,  i.e. in the components  $\mathscr{H}^{\text{hyp}}(6)$ and $\mathscr{H}^{\text{hyp}}(3,3)$. However, an infinite family is yet to be found. Below, we also include other examples in genus 3.

\subsubsection{Five genus 3 examples in $\mathscr{H}^{\text{hyp}}(4)$}   We look further into the Rauzy hyperelliptic diagram with 6 subintervals. The following examples start at the hyperelliptic pair $\pi=(6,5,4,3,2,1)$. We will give the paths and the characteristic polynomial of the associated matrix. Hence, we can see the loops produce SAF-zero pseudo-Anosov from the cubic factors of the characteristic polynomials.

\begin{itemize}
  \item $11010100111000010;\\ 
(x^3-6x^2+5x-1)(x^3-5x^2+6x-1)$ 
  \item $1110101100010100;\\
(x^3-6x^2+5x-1)(x^3-5x^2+6x-1)$ 
  \item $11010000100111000010;\\
(x^3-8x^2+6x-1)(x^3-6x^2+8x-1)$ 
  \item $11010100111000011110;\\
(x^3-8x^2+6x-1)(x^3-6x^2+8x-1)$ 
  \item $1110111101100010100;\\
(x^3-8x^2+6x-1)(x^3-6x^2+8x-1)$.
\end{itemize}
\subsubsection{Seven genus 4 examples in $\mathscr{H}^{\text{hyp}}(6)$} We look into the Rauzy hyperelliptic diagram with 8 subintervals. The following examples start at the hyperelliptic pair $\pi=(8,7,6,5,4,3,2,1)$. We will give the paths and the characteristic polynomial of the associated matrix. Hence, we can see the loops produce SAF-zero pseudo-Anosov from the cubic factors of the characteristic polynomials.
\begin{itemize}
  \item $101010101100011110000101111110000000;\\ 
(x^3-9x^2+6x-1)(x^3-6x^2+9x-1)(x-1)^2$ 
  \item $101010101100011101000001111110000000;\\
(x^3-9x^2+6x-1)(x^3-6x^2+9x-1)(x-1)^2$ 
  \item $1101010100111000011110100000010;\\
(x^3-9x^2+6x-1)(x^3-6x^2+9x-1)(x-1)^2$ 
  \item $1101010100111000101111100000010;\\
(x^3-9x^2+6x-1)(x^3-6x^2+9x-1)(x-1)^2$ 
  \item $11010100100111000011111000000110;\\
(x^2-6x+1)(x^3-6x^2+5x-1)(x^3-5x^2+6x-1)$
  \item $11101011011000111100000100100;\\
(x^2-6x+1)(x^3-6x^2+5x-1)(x^3-5x^2+6x-1)$
  \item $1111010010011100001011011000;\\
(x^2-6x+1)(x^3-6x^2+5x-1)(x^3-5x^2+6x-1)$.
\end{itemize}

\subsubsection{Three genus 4 examples in  $\mathscr{H}^{\text{hyp}}(3,3)$}   We look into the Rauzy hyperelliptic diagram with 9 subintervals. The following examples start at the hyperelliptic pair $\pi=(9,8,7,6,5,4,3,2,1)$. We will give the paths and the characteristic polynomial of the associated matrix. Hence, we can see the loops produce SAF-zero pseudo-Anosov from the non-reciprocal quartic factors of the characteristic polynomials.
\begin{itemize}
  \item $11010101101100011110000011111010000000110;\\ 
  (x-1)(x^4-9x^3+22x^2-11x+1)(x^4-11x^3+22x^2-9x+1)$ 
  \item $11010101101100011101000001111110000000110;\\
  (x-1)(x^4-9x^3+22x^2-11x+1)(x^4-11x^3+22x^2-9x+1)$ 
  \item $11110101101100011101000001011011000;\\
  (x-1)(x^4-9x^3+22x^2-11x+1)(x^4-11x^3+22x^2-9x+1)$. 
\end{itemize}

\end{document}